\documentclass[12pt]{amsart}

\usepackage[utf8]{inputenc}
\usepackage{a4}

\usepackage{amsmath}
\usepackage{amssymb}
\usepackage{stmaryrd}
\usepackage{graphicx}
\usepackage{bbm}
\usepackage{rotating}

\title{Extended Formulations for Packing and Partitioning Orbitopes}
\author[Faenza]{Yuri Faenza}
\address[Faenza]{Dipartimento di Ingegneria dell'Impresa, Università di Roma "Tor Vergata", Rome,
Italy}
\email[Faenza]{faenza@disp.uniroma2.it}
\thanks{This work has been supported by the European Union, FP6, MRTN-CT-2003-504438 (ADONET)}
\author[Kaibel]{Volker Kaibel}
\address[Kaibel]{Otto-von-Guericke Universit\"at Magdeburg, Fakult\"at f\"ur Mathematik, Universit\"atsplatz~2, 39106~Magdeburg, Germany}
\email[Kaibel]{kaibel@ovgu.de}

\date{\today}

\newtheorem{theorem}{Theorem}
\newtheorem{lemma}[theorem]{Lemma}
\newtheorem{rem}[theorem]{Remark}
\newtheorem{corollary}[theorem]{Corollary}

\newcommand{\R}{\mathbbm{R}}
\newcommand{\orbipartinds}[2]{\mathcal{I}_{{#1},{#2}}}
\DeclareMathOperator{\barop}{bar}
\DeclareMathOperator{\bigoop}{O}
\newcommand{\barr}[2]{\barop_{{#1},{#2}}}
\DeclareMathOperator{\orbiop}{O}
\newcommand{\orbipack}[2]{\orbiop^{\le}_{#1,#2}}
\newcommand{\orbipart}[2]{\orbiop^{=}_{#1,#2}}
\newcommand{\zerovec}{\mathbf{0}}
\newcommand{\onevec}{\mathbf{1}}
\newcommand{\row}[1]{\text{row}_{#1}}
\newcommand{\bigo}[1]{\bigoop({#1})}

\DeclareMathOperator{\digraphOp}{D}
\newcommand{\digraph}[2]{\digraphOp_{{#1},{#2}}}
\DeclareMathOperator{\nodesOp}{V}
\newcommand{\nodes}[2]{\nodesOp_{{#1},{#2}}}
\DeclareMathOperator{\arcsOp}{A}
\newcommand{\arcs}[2]{\arcsOp_{{#1},{#2}}}
\newcommand{\vertarcs}[2]{\arcsOp^{\shortdownarrow}_{{#1},{#2}}}
\newcommand{\varc}[2]{({#1},{#2})^{\shortdownarrow}}
\newcommand{\diagarcs}[2]{\arcsOp^{\begin{rotate}{45}\tiny{$\shortdownarrow$}\end{rotate}}_{{#1},{#2}}}
\newcommand{\darc}[2]{({#1},{#2})^{\,\begin{rotate}{45}\tiny{$\shortdownarrow$}\end{rotate}}}
\newcommand{\ints}[1]{[{#1}]}
\newcommand{\intszero}[1]{[{#1}]_0}
\newcommand{\setdef}[2]{\{{#1}\,:\,{#2}\}}
\newcommand{\nodeset}[1]{\nodesOp({#1})}
\DeclareMathOperator{\tnodesOp}{T}
\newcommand{\tnodeset}[1]{\tnodesOp({#1})}
\DeclareMathOperator{\snodesOp}{S}
\newcommand{\snodeset}[1]{\snodesOp({#1})}
\DeclareMathOperator{\flowsOp}{F}
\newcommand{\flows}[2]{\flowsOp_{{#1},{#2}}}
\DeclareMathOperator{\colop}{col}
\newcommand{\col}[1]{\colop_{#1}}
\newcommand{\vbar}[2]{\col{{#1},{#2}}}
\DeclareMathOperator{\outOp}{out}
\newcommand{\outarcs}[1]{\outOp({#1})}
\newcommand{\vertoutarcs}[1]{\outOp^{\shortdownarrow}({#1})}
\DeclareMathOperator{\inOp}{in}
\newcommand{\inarcs}[1]{\inOp({#1})}
\newcommand{\vertinarcs}[1]{\inOp^{\shortdownarrow}({#1})}
\newcommand{\diaginarcs}[1]{\inOp^{\begin{rotate}{45}\tiny{$\shortdownarrow$}\end{rotate}}\,({#1})}
\DeclareMathOperator{\extpolyOp}{P}
\newcommand{\extpoly}[2]{\extpolyOp_{{#1},{#2}}}
\newcommand{\extpolycompact}[2]{\extpolyOp^{\text{comp}}_{{#1},{#2}}}
\newcommand{\extpolypart}[2]{\extpolyOp^{=}_{{#1},{#2}}}
\DeclareMathOperator{\scipolyOp}{Q}
\newcommand{\scipoly}[2]{\scipolyOp_{{#1},{#2}}}
\newcommand{\scalprod}[2]{\langle{#1},{#2}\rangle}

\graphicspath{{Figures/}}

\begin{document}

\maketitle

\begin{abstract}
    We give  compact extended formulations for the packing and partitioning orbitopes (with respect to the full symmetric group) described and analyzed in~\cite{KP08}. These polytopes are the convex hulls of all 0/1-matrices with lexicographically sorted columns and at most, resp. exactly, one $1$-entry per row. They are important objects for symmetry reduction in certain integer programs.
    Using the extended formulations, we also derive a rather simple proof of the fact~\cite{KP08} that basically shifted-column inequalities suffice in order to describe those orbitopes linearly.
\end{abstract}

\section{Introduction}\label{sec:intro}

Exploitation of symmetries is crucial for many very difficult integer programming models.
Over the last few years, significant progress has been achieved with respect to general techniques for dealing with symmetries within branch-and-cut algorithms.
Very nice and effective procedures have been devised, like isomorphism pruning \cite{Mar02,Mar03,Mar03b,Mar07} and orbital branching~\cite{LiOsRoSm06,LiOsRoSm08}. There has also been progress in understanding linear inequalities to be added to certain integer programs in order to remove symmetry. Towards this end, orbitopes have been introduced in~\cite{KP08}.

The \emph{packing orbitope} $\orbipack{p}{q}$ and the \emph{partitioning orbitope} $\orbipart{p}{q}$ are the convex hulls of all 0/1-matrices  of size $p\times q$ whose columns are in lexicographically decreasing order having
 at most or exactly, respectively,  one $1$-entry per row. In~\cite{KP08}, complete  descriptions with linear inequalities have been derived for these polytopes (see Thm.~16 and~17 in~\cite{KP08}).
Knowledge on  orbitopes turns out to be quite useful in practical symmetry reduction for certain integer programming models. For instance, in a well-known  formulation of the graph partitioning problem (for  graphs having~$p$ nodes to be partitioned into~$q$ parts) the symmetry on the 0/1-variables $x_{ij}$ indicating whether node~$i$ is put into part~$j$ of the partitioning the symmetry arising from permuting the parts can be removed by requiring~$x\in\orbipart{p}{q}$. We refer to~\cite{KaPePf07} and~\cite{KP08} for a more detailed discussion of the practical use of orbitopes.

The topic of this paper are extended formulations for these
orbitopes, i.e., (simple) linear descriptions of  higher
dimensional polytopes which can be projected to~$\orbipack{p}{q}$
and~$\orbipart{p}{q}$. In fact, such extended formulations play
important roles in polyhedral combinatorics and integer
programming in general, because rather than solving a linear
optimization problem over a polyhedron in the original space, one
may solve it over a (hopefully simpler described) polyhedron of
which the first one is a linear projection. For instance, the
lift-and-project approach~\cite{BCC93} and other general
reformulation schemes (e.g., the ones due to Lov\'{a}sz and
Schrijver~\cite{LS91} as well as Adams and Sherali~\cite{SA94})
are based on extended formulations. Recent examples include work
on mixed integer programming for duals of network
matrices~\cite{CdSW08,CGZ07,CdSEW06}. More classical is the
general theory on extended formulations obtained from (certain)
dynamic programming algorithms~\cite{MRC90}. The results of this
paper are much in the spirit of the latter work.

In order to give an overview on the contributions of this paper let us first
recall a few facts on orbitopes. As no 0/1-matrix with at most one
$1$-entry per row and lexicographically (decreasing) sorted
columns  has a one above its main diagonal, we may assume without loss of generality
that
$\orbipart{p}{q}\subseteq\orbipack{p}{q}\subseteq\R^{\orbipartinds{p}{q}}$
with
$\orbipartinds{p}{q}=\setdef{(i,j)\in\ints{p}\times\ints{q}}{i\ge
j}$ (where
$\ints{n}=\{1,2,\dots,n\}$). In fact, $\orbipart{p}{q}$ is the
face of~$\orbipack{p}{q}$ defined by requiring that all
\emph{row-sum inequalities} $x(\row{i})\le 1$ for  $i\in\ints{p}$
are satisfied with equality, where
$\row{i}=\setdef{(i,j)\in\orbipartinds{p}{q}}{j\in\ints{q}}$.

The main result of~\cite{KP08} is a complete  description
of~$\orbipack{p}{q}$ and~$\orbipart{p}{q}$ by means of linear
equations and inequalities. This system of constraints (the
\emph{SCI-system}) consists, next to  nonnegativity constraints
and  row-sum inequalities (or row-sum equations for
$\orbipart{p}{q}$), of the  exponentially large class  of
\emph{shifted column inequalities} (\emph{SCI}) that will be
defined  at the end of Sect.~\ref{sec:setup}. In~\cite{KP08} it is
also proved that, up to a few exceptions, these exponentially many
SCIs define facets of the orbitopes. The proof given
in~\cite{KP08} of the fact that the SCI-system completely
describes these orbitopes is rather lengthy and somewhat
technical. Extending over pages $18$ to $27$, it hardly leaves (not only) the
reader with a good idea of the reasons
for the SCI-system being sufficient to describe the orbitopes.

In contrast to this, the contributions of the present work are the
following: We provide a quite simple extended formulation
for~$\orbipack{p}{q}$ (along with a rather short proof
establishing this) and, moreover, we show by some simple and
natural (not technical) arguments that the SCI-system describes
the projection of the feasible region of that extended formulation
to the original space, thus providing a new proof showing that the
SCI-system describes~$\orbipack{p}{q}$.

This latter proof is much shorter than the original one, and it seems to provide much better insight into the  reasons for the SCI-system to describe the orbitopes. Clearly, as $\orbipart{p}{q}$ is a face of~$\orbipack{p}{q}$, the results for the latter polytope immediately yield corresponding results for the first one.
However, besides leading to that simpler proof, we believe that our extended formulation for $\orbipack{p}{q}$ is interesting itself. It provides a description of a quite natural polytope (the orbitope~$\orbipack{p}{q}$) by a system of constraints in a space whose dimension is roughly twice the original dimension~$|\orbipartinds{p}{q}|$ with only linearly (in~$|\orbipartinds{p}{q}|$) many nonzero coefficients, while every linear description of the orbitope in the original space requires exponentially many inequalities. This may also turn out to be  computationally attractive.

The basic idea of our extended formulation is to assign to each vertex of~$\orbipack{p}{q}$ a directed path in a certain acyclic digraph. The additional variables in our extended formulations are used to suitably express these paths. The
digraph we work with is set up in Sect.~\ref{sec:setup}, where we also fix some notations and define SCIs.
In Sect.~\ref{sec:ext} we then describe the extended formulations for $\orbipack{p}{q}$ and $\orbipart{p}{q}$ (Thm.~\ref{thm:extform} and Cor.~\ref{cor:extpartpath}).
 The
main work is done in Sect.~\ref{subsec:ext:pack}, where the
extended formulation for~$\orbipack{p}{q}$ with additional
variables encoding the paths mentioned above is introduced and
proved to define an integral polyhedron (Thm.~\ref{thm:integral}).
From this it is easy to conclude that the formulation indeed
defines a polytope that projects down to~$\orbipack{p}{q}$
(Thm.~\ref{thm:extform}). Both the extension of such results to
the partitioning case $\orbipart{p}{q}$
(Cor.~\ref{cor:extpartpath} in Sect.~\ref{subsec:ext:part}), and
the transformations of the systems in order to reduce the numbers
of variables and nonzero coefficients (Thm.~\ref{thm:verycompact}
in Sect.~\ref{subsec:ext:reduce}) are obtained without much work.
 On the way, we also derive linear (in~$|\orbipartinds{p}{q}|$) time algorithms for optimizing linear objective functions over~$\orbipack{p}{q}$ and~$\orbipart{p}{q}$ (Cor.~\ref{cor:alg} and~\ref{cor:alg:part}).
In Sect.~\ref{sec:project} we finally prove that  the projection
of the feasible region defined by the extended formulation contains the polytope defined by
the SCI-system
(Thm.~\ref{thm:liftscipoly}), thus providing the new  proof of the fact (Thm.~\ref{thm:orbipackSCI})
that the latter polytope equals~$\orbipack{p}{q}$. We conclude with a few remarks and
acknowledgements in Sect.~\ref{sec:remarks}.

\section{The Setup}
\label{sec:setup}

Let us assume $p\ge q\ge 1$ throughout the paper.
We define a directed acyclic graph $\digraph{p}{q}=(\nodes{p}{q},\arcs{p}{q})$ with node set
\[
    \nodes{p}{q}=\orbipartinds{p}{q}\uplus(\intszero{p}\times\{0\})\uplus\{s\}\uplus\{t\}
\]
(where $\intszero{n}=\ints{n}\cup\{0\}$ and $\uplus$ means disjoint union).
Using the notation $q(i)=\min\{i,q\}$,
the set of arcs of $\digraph{p}{q}$ is
\[
    \arcs{p}{q}=\vertarcs{p}{q} \cup\diagarcs{p}{q}\cup\{(s,(0,0))\}\cup\setdef{((p,j),t)}{j\in\intszero{q}}\,,
\]
where
\[
    \vertarcs{p}{q}=\setdef{((i,j),(i+1,j))}{i\in\intszero{p-1},j\in\intszero{q(i)}}
\]
is the set of \emph{vertical arcs} that are denoted by
$\varc{i}{j}=((i,j),(i+1,j))$, and
\[
    \diagarcs{p}{q}=\setdef{((i,j),(i+1,j+1))}{i\in\intszero{p-1},j\in\intszero{q(i+1)-1}}\,
\]
is the set of \emph{diagonal arcs} that are denoted by
$\darc{i}{j}=((i,j),(i+1,j+1))$.
The crucial property of~$\digraph{p}{q}$ is that every vertex of~$\orbipack{p}{q}$ induces an $s$-$t$-path in~$\digraph{p}{q}$ as indicated in Fig.~\ref{fig:digraph}. Note that different vertices may induce the same path.

\begin{figure}
    \centering
    \includegraphics[width=.35\textwidth]{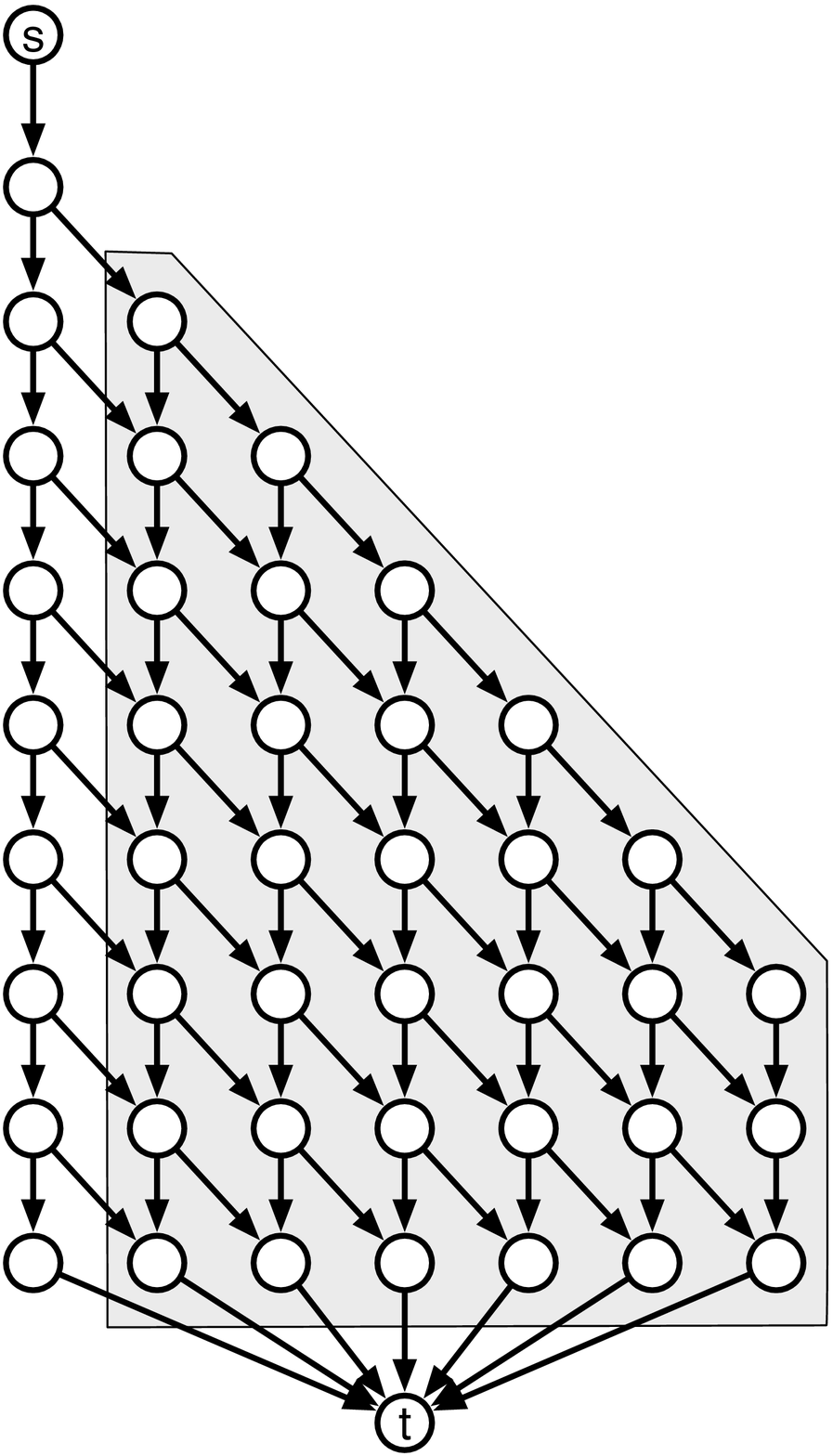}
    \hspace{.1\textwidth}
    \includegraphics[width=.35\textwidth]{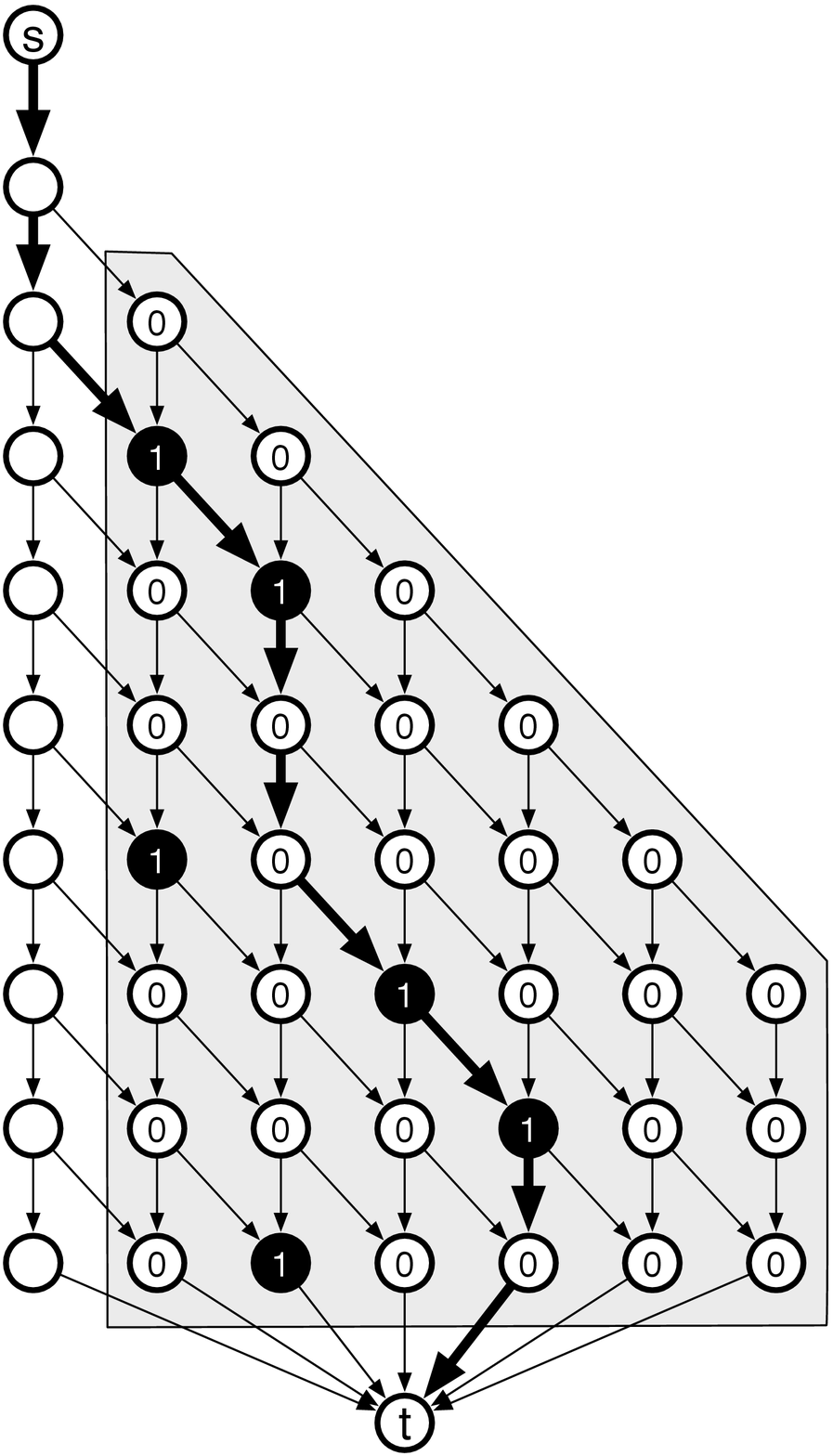}
    \caption{The digraph~$\digraph{8}{6}$ and a vertex of $\orbipack{8}{6}$ along with its $s$-$t$-path.}\label{fig:digraph}
\end{figure}

For a subset $W\subseteq\nodes{p}{q}$ we use the following notation:
\begin{eqnarray*}
    \outarcs{W} & = & \setdef{(w,u)\in\arcs{p}{q}}{w\in W,u\not\in W}\\
    \vertoutarcs{W} & = & \setdef{(w,u)\in\vertarcs{p}{q}}{w\in W,u\not\in W}\\
    \inarcs{W} & = & \setdef{(u,w)\in\arcs{p}{q}}{w\in W,u\not\in W}\\
    \vertinarcs{W} & = & \setdef{(u,w)\in\vertarcs{p}{q}}{w\in W,u\not\in W}\\
    \diaginarcs{W} & = & \setdef{(u,w)\in\diagarcs{p}{q}}{w\in W,u\not\in W}
\end{eqnarray*}
For a directed path~$\Gamma$ in $\digraph{p}{q}$, we denote by

\begin{tabular}{ll}
    $\nodeset{\Gamma}\subseteq\nodes{p}{q}$ & the set of nodes on the path~$\Gamma$,\\
    $\snodeset{\Gamma}\subseteq\nodeset{\Gamma}$ & the set of nodes on~$\Gamma$ not entered  by~$\Gamma$ via  diagonal arcs, and\\
    $\tnodeset{\Gamma}\subseteq\nodeset{\Gamma}$ & the set of nodes on~$\Gamma$  left by~$\Gamma$  via  diagonal arcs
\end{tabular}
(see Fig.~\ref{fig:stnodes}).
\begin{figure}
    \centering
    \includegraphics[width=.35\textwidth]{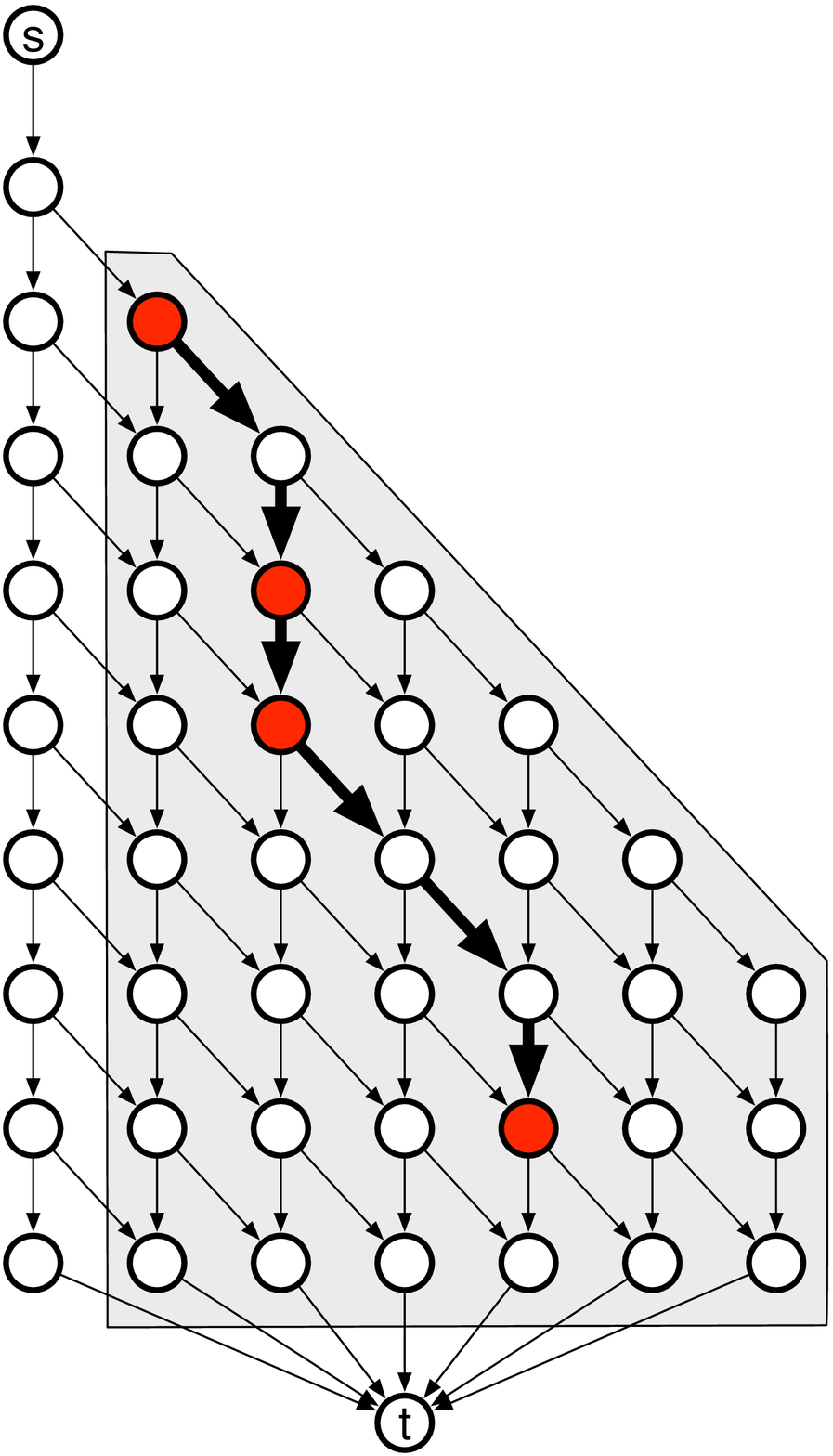}
    \hspace{.1\textwidth}
    \includegraphics[width=.35\textwidth]{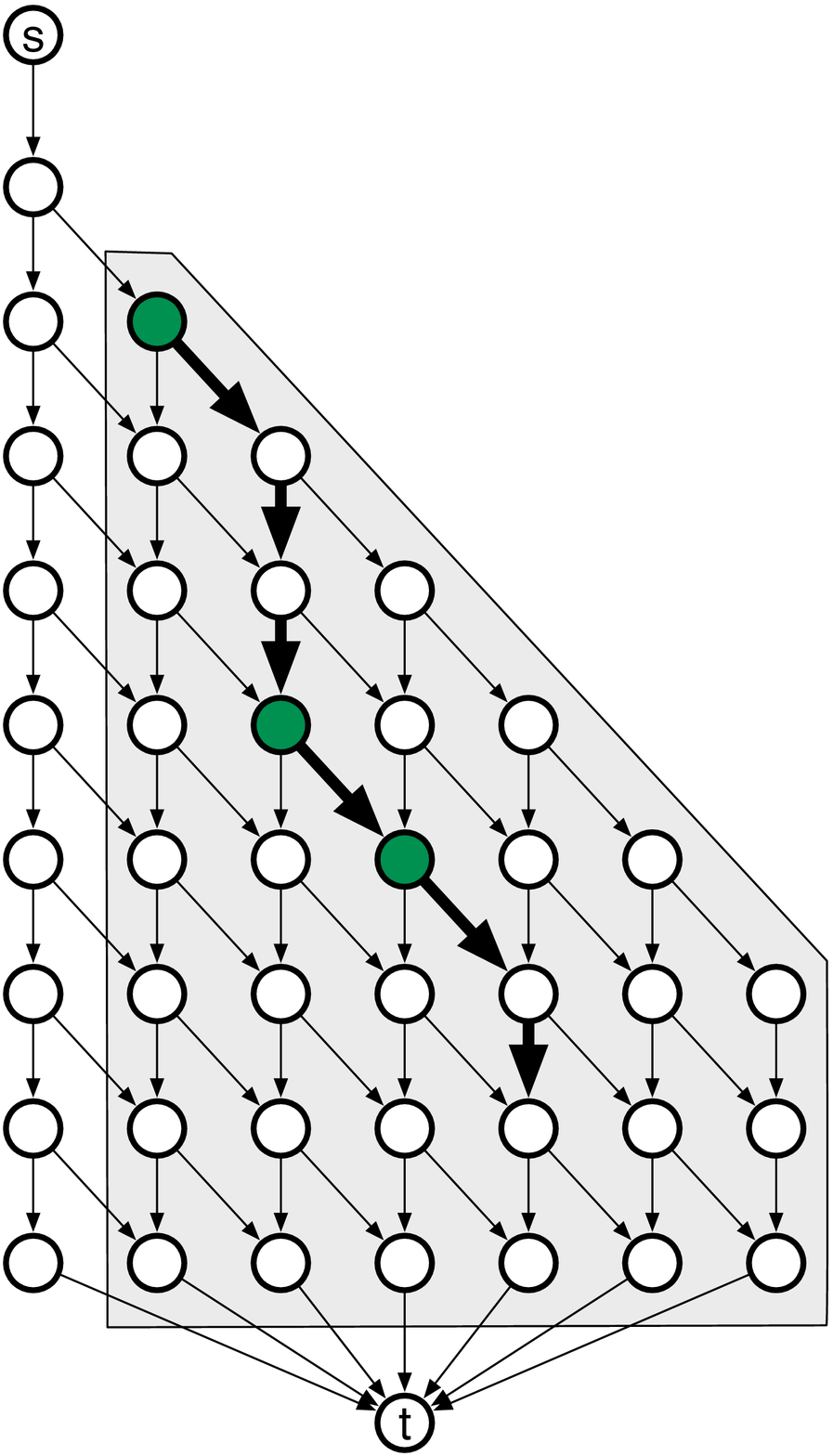}
    \caption{A path~$\Gamma$ along with the sets $\snodeset{\Gamma}$ (left) and $\tnodeset{\Gamma}$ (right).}\label{fig:stnodes}
\end{figure}
Note that $\snodeset{\Gamma}$ always contains the start node of~$\Gamma$, and $\tnodeset{\Gamma}$ always excludes the end node of~$\Gamma$.

\begin{rem}\label{rem:inoutgamma}
    For every directed path~$\Gamma$ in $\digraph{p}{q}$ with end node~$(i,j)\in\orbipartinds{i}{j}$, we have
    \[
        \diaginarcs{\nodeset{\Gamma}}=\diaginarcs{\snodeset{\Gamma}}
        \quad\text{and}\quad
        \vertoutarcs{\nodeset{\Gamma}\setminus\{(i,j)\}}=\vertoutarcs{\tnodeset{\Gamma}}\,.
    \]
\end{rem}

A subset $S\subseteq\orbipartinds{p}{q}$ is a \emph{shifted
column} if and only if $S=\snodeset{\Gamma}$ for some
$(\ell,\ell)$-$(i-1,j-1)$-path~$\Gamma$ in $\digraph{p}{q}$ with
$i\in\ints{p}\setminus\{1\}$, $j\in\ints{q}\setminus\{1\}$, and
$\ell\in\ints{q}$. The associated \emph{shifted-column
inequality} is $x(\barr{i}{j})\le x(S)$, where
$\barr{i}{j}=\setdef{(i,\ell)\in\orbipartinds{p}{q}}{\ell \ge j}$
and, as usual, we write $z(N)=\sum_{e\in N}z_e$ for some vector
$z\in \R^{M}$ and a subset $N\subseteq M$ (see
Fig.~\ref{fig:sci}).
\begin{figure}
    \centering
    \includegraphics[width=.35\textwidth]{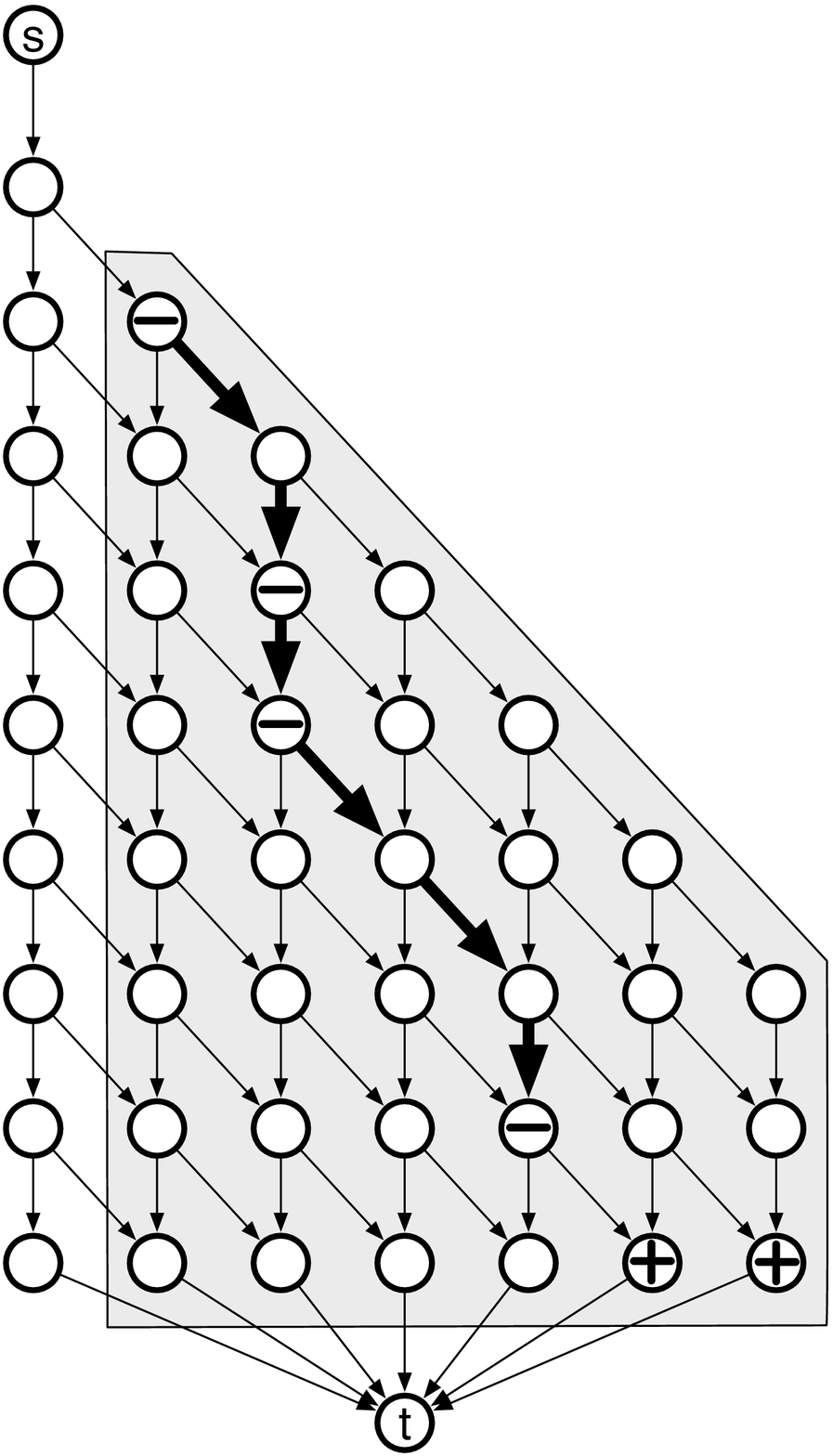}
    \hspace{.1\textwidth}
    \includegraphics[width=.35\textwidth]{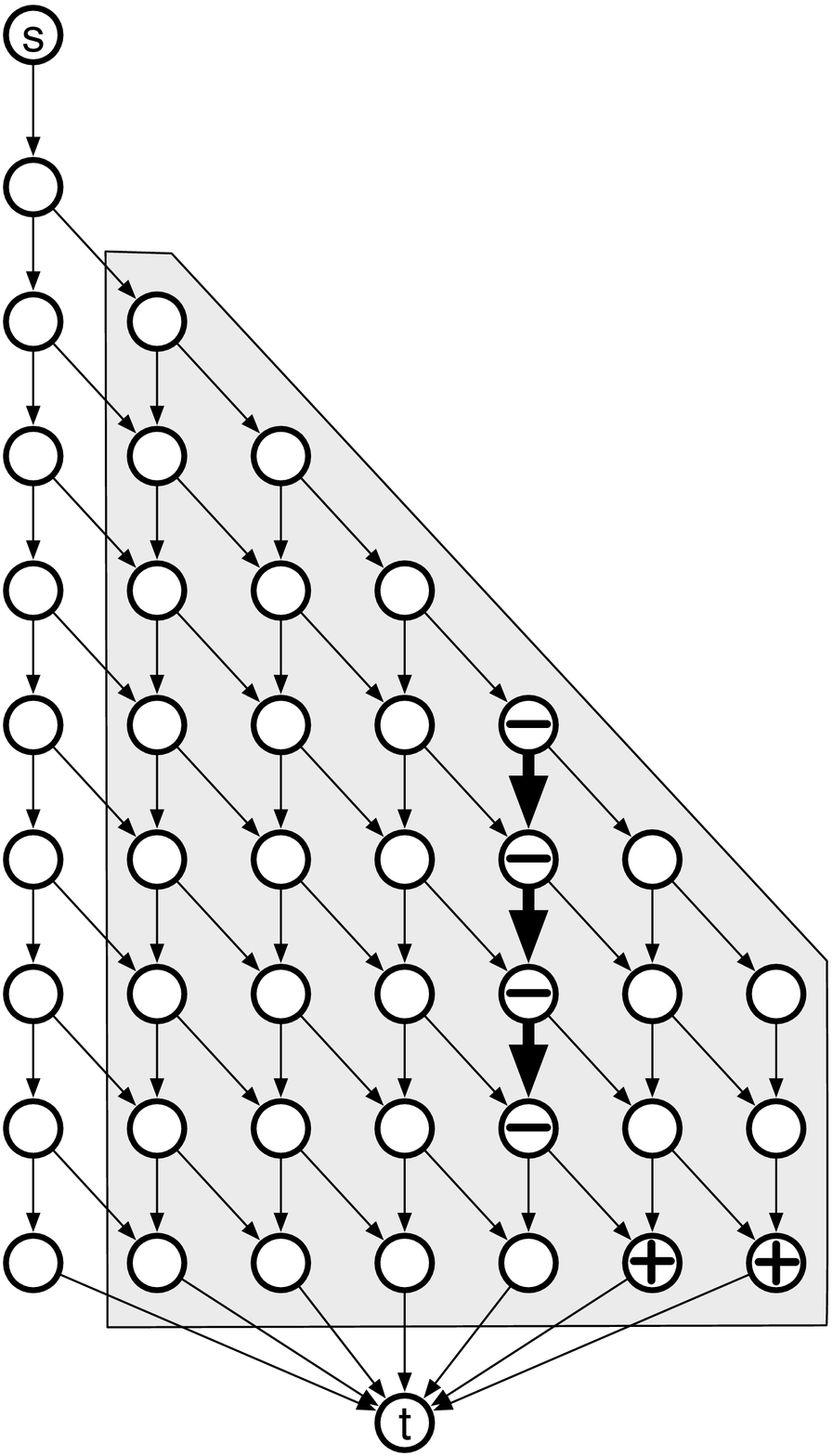}
    \caption{Coefficient vectors of two SCIs with the same bar $\barr{8}{5}$.}\label{fig:sci}
\end{figure}

\section{Extended formulations}
\label{sec:ext}

\subsection{The packing case}
\label{subsec:ext:pack}
Denote by $\flows{p}{q}\subseteq\R^{\arcs{p}{q}}$ the set of all $s$-$t$-flows (without any capacity restrictions)  in $\digraph{p}{q}$ with flow value one. Clearly, $\flows{p}{q}$ is an integral polytope. Since $\digraph{p}{q}$ is acyclic, the vertices of $\flows{p}{q}$ are the incidence vectors of the directed $s$-$t$-paths (viewed as subsets of arcs) in $\digraph{p}{q}$.

For a flow~$y\in\flows{p}{q}$ and a node $(i,j)\in\orbipartinds{p}{q}$, we denote by
\[
    y(i,j)=y(\inarcs{i,j})=y(\outarcs{i,j})
\]
the amount of flow passing node $(i,j)$. For a subset $W\subseteq\row{i}$ of nodes in the same row, $y(W)=\sum_{w\in W}y(w)$ is the total amount of flow entering~$W$ (or, equivalently, leaving~$W$).

\begin{lemma}\label{lem:cuts}
    For a directed $(k,\ell)$-$(i,j)$-path~$\Gamma$ in $\digraph{p}{q}$ with $(i,j)\in\orbipartinds{p}{q}$ the following statements hold for all $y\in\flows{p}{q}$ (see Fig.~\ref{fig:lemma2}):
    \begin{enumerate}
        \item If $k=\ell\ge 1$ then
        \[
        y(\diaginarcs{\snodeset{\Gamma}})-y(\vertoutarcs{\tnodeset{\Gamma}})=y(\barr{i}{j})\,.
        \]
        \item If $\ell=0$ then
        \[
        1-y(\vertoutarcs{\tnodeset{\Gamma}})=y(\barr{i}{j})\,.
        \]
    \end{enumerate}
\end{lemma}
\begin{figure}
    \centering
    \includegraphics[width=.35\textwidth]{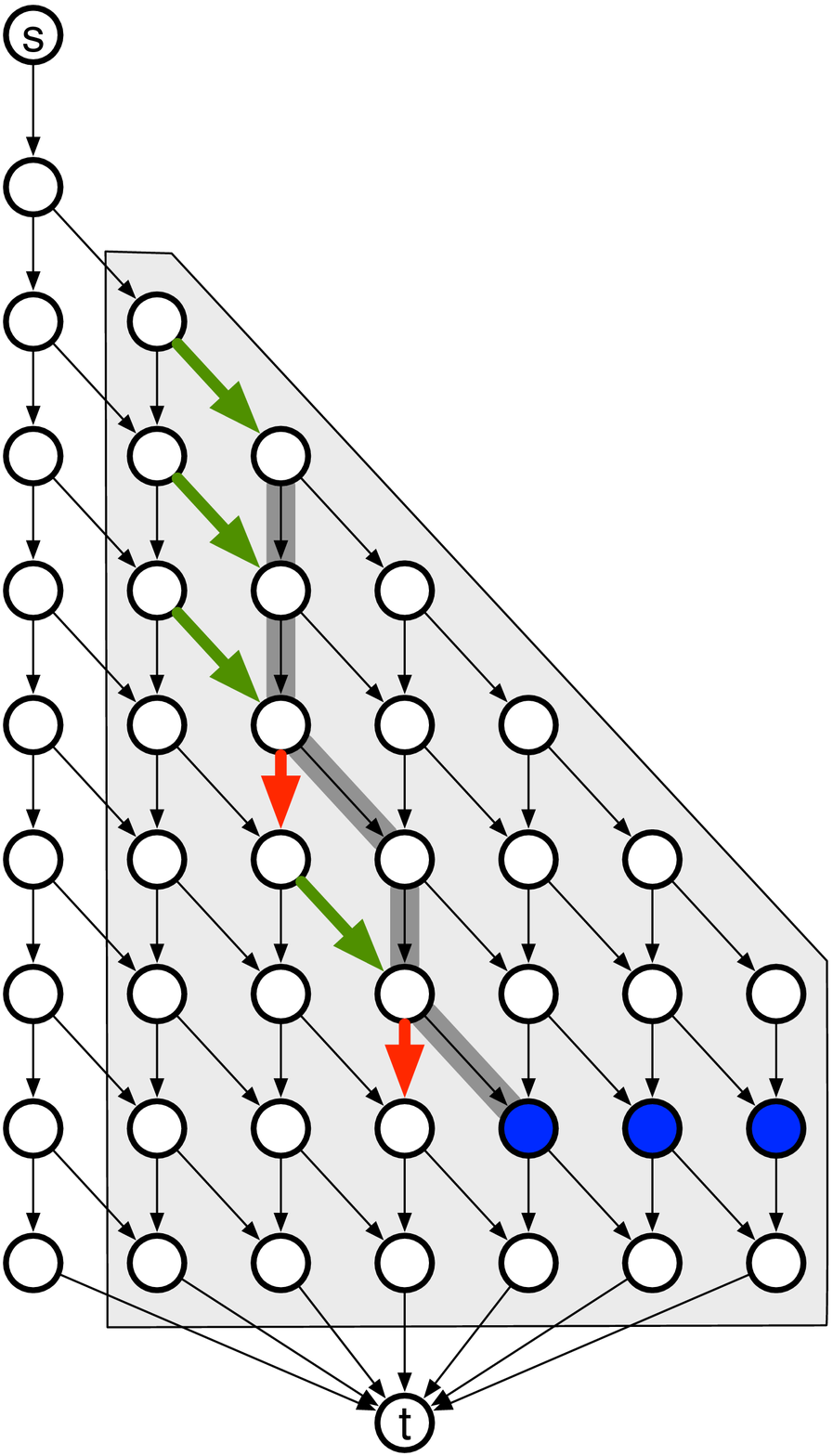}
    \hspace{.1\textwidth}
    \includegraphics[width=.35\textwidth]{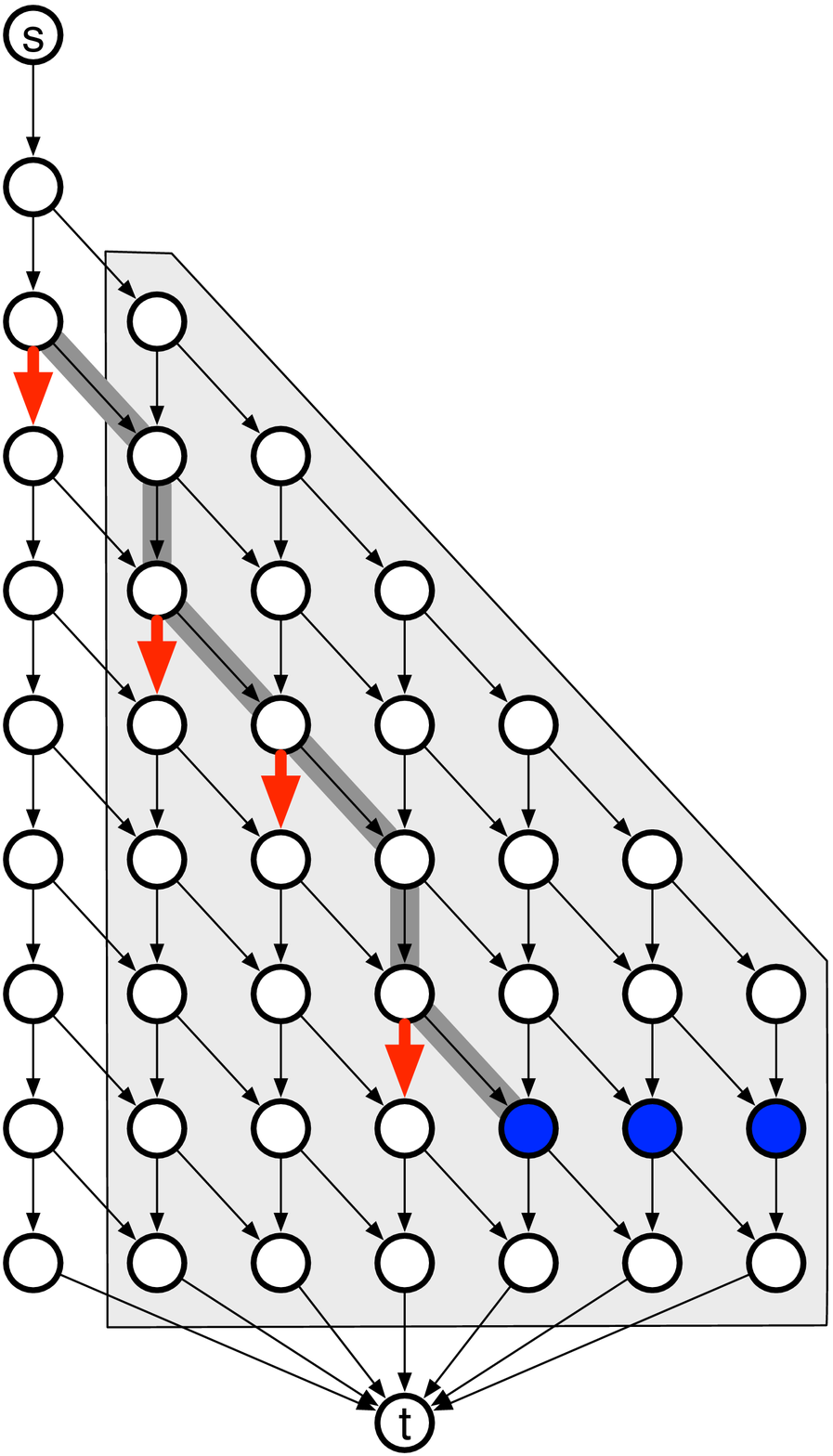}
    \caption{Illustration of part~(1) (left) and part~(2) (right) of Lemma~\ref{lem:cuts} with $(i,j)=(7,4)$.}\label{fig:lemma2}
\end{figure}

\begin{proof}
    For both cases, let
    \[
        W=\setdef{(a',b)\in\nodes{p}{q}\setminus\{s,t\}}{a'\le a \text{ for some }(a,b)\in\nodeset{\Gamma}\cup\barr{i}{j}}
    \]
    be the set of all nodes (different from~$s$ and~$t$) in or above $\nodeset{\Gamma}\cup\barr{i}{j}$.

    For case~(1), we start by observing
    \begin{equation}\label{eq:lem:cuts:1}
        y(\inarcs{W})=y(\outarcs{W})
    \end{equation}
    (as~$y\in\flows{p}{q}$ is an $s$-$t$-flow).
    Because of $(0,0)\not\in W$ (thus~$s$ being not adjacent to~$W$) we have $\inarcs{W}=\diaginarcs{\nodeset{\Gamma}}$, which according to Remark~\ref{rem:inoutgamma} equals $\diaginarcs{\snodeset{\Gamma}}$, yielding
    \begin{equation}\label{eq:lem:cuts:2}
        y(\inarcs{W})=y(\diaginarcs{\snodeset{\Gamma}})\,.
    \end{equation}
    Similarly, we have $\outarcs{W}=\vertoutarcs{\nodeset{\Gamma}\setminus\{(i,j)\}}\uplus\outarcs{\barr{i}{j}}$, where Remark~\ref{rem:inoutgamma} gives  $\vertoutarcs{\nodeset{\Gamma}\setminus\{(i,j)\}}=\vertoutarcs{\tnodeset{\Gamma}}$. Thus, we obtain
    \begin{equation}\label{eq:lem:cuts:3}
    y(\outarcs{W})  =y(\vertoutarcs{\tnodeset{\Gamma}})+y(\outarcs{\barr{i}{j}})\,.
    \end{equation}
    Equations~\eqref{eq:lem:cuts:1}, \eqref{eq:lem:cuts:2}, and~\eqref{eq:lem:cuts:3} imply the statement on case~(1).

    For case~(2), we exploit the fact that the $s$-$t$-flow $y\in\flows{p}{q}$ of value one satisfies
    \begin{equation}\label{eq:cuts:case2}
        1 = y(\outarcs{W\cup\{s\}})-y(\inarcs{W\cup\{s\}})\,.
    \end{equation}
    We have $\inarcs{W\cup\{s\}}=\varnothing$ and  $\outarcs{W\cup\{s\}}=\outarcs{W}$ (due to $(0,0)\in W$ in this case). From
    \[
        \outarcs{W}=\vertoutarcs{\nodeset{\Gamma}\setminus\{(i,j)\}}
        \uplus
        \outarcs{\barr{i}{j}}
    \]
    and $\vertoutarcs{\nodeset{\Gamma}\setminus\{(i,j)\}}=\vertoutarcs{\tnodeset{\Gamma}}$ (see Remark~\ref{rem:inoutgamma}), we thus derive the statement on case~(2) from~\eqref{eq:cuts:case2}.
\end{proof}

For $(i,j)\in\orbipartinds{p}{q}$ we denote by
\[
    \vbar{i}{j}=\setdef{(k,j)}{j\le k\le i}
\]
the upper part of the $j$-th column from $(j,j)$ down to $(i,j)$, including both nodes. For the directed path~$\Gamma$ with $\nodeset{\Gamma}=\vbar{i}{j}$, we have
$\snodeset{\Gamma}=\vbar{i}{j}$ and
$\tnodeset{\Gamma}=\varnothing$. Thus, part~(1) of Lemma~\ref{lem:cuts} implies the following.
\begin{rem}\label{rem:cuts}
    For all $y\in\flows{p}{q}$ and $(i,j)\in\orbipartinds{p}{q}$, we have
    \[
        y(\diaginarcs{\vbar{i}{j}})=y(\barr{i}{j})\,.
    \]
\end{rem}

The central object of study of this paper is the polytope
\[
    \extpoly{p}{q}=
    \setdef{(x,y)\in\R^{\orbipartinds{p}{q}}\times\R^{\arcs{p}{q}}}%
           {y\in\flows{p}{q} \text{ and }(x,y)%
            \text{ satisfies~\eqref{eq:bindxy} and~\eqref{eq:bindxybar} below}}
\]
with
\begin{equation}\label{eq:bindxy}
    y_{\darc{i-1}{j-1}}\le x_{ij}\quad\text{ for all }(i,j)\in\orbipartinds{p}{q}
\end{equation}
and
\begin{equation}\label{eq:bindxybar}
    x(\barr{i}{j})\le y(\diaginarcs{\vbar{i}{j}})\quad\text{ for all }(i,j)\in\orbipartinds{p}{q}
\end{equation}
(see Fig.~\ref{fig:eq56}).
\begin{figure}
    \centering
    \includegraphics[width=.35\textwidth]{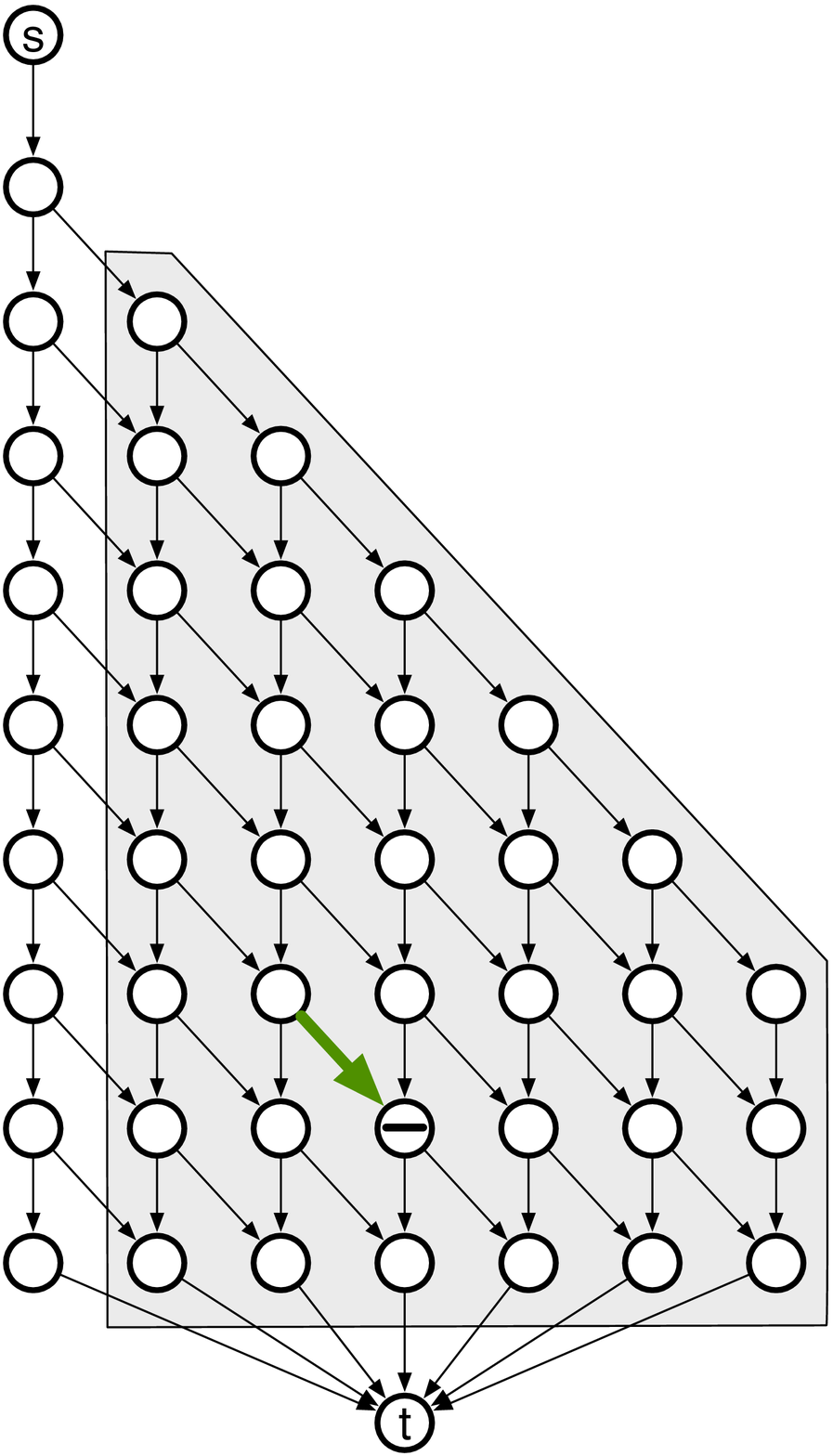}
    \hspace{.1\textwidth}
    \includegraphics[width=.35\textwidth]{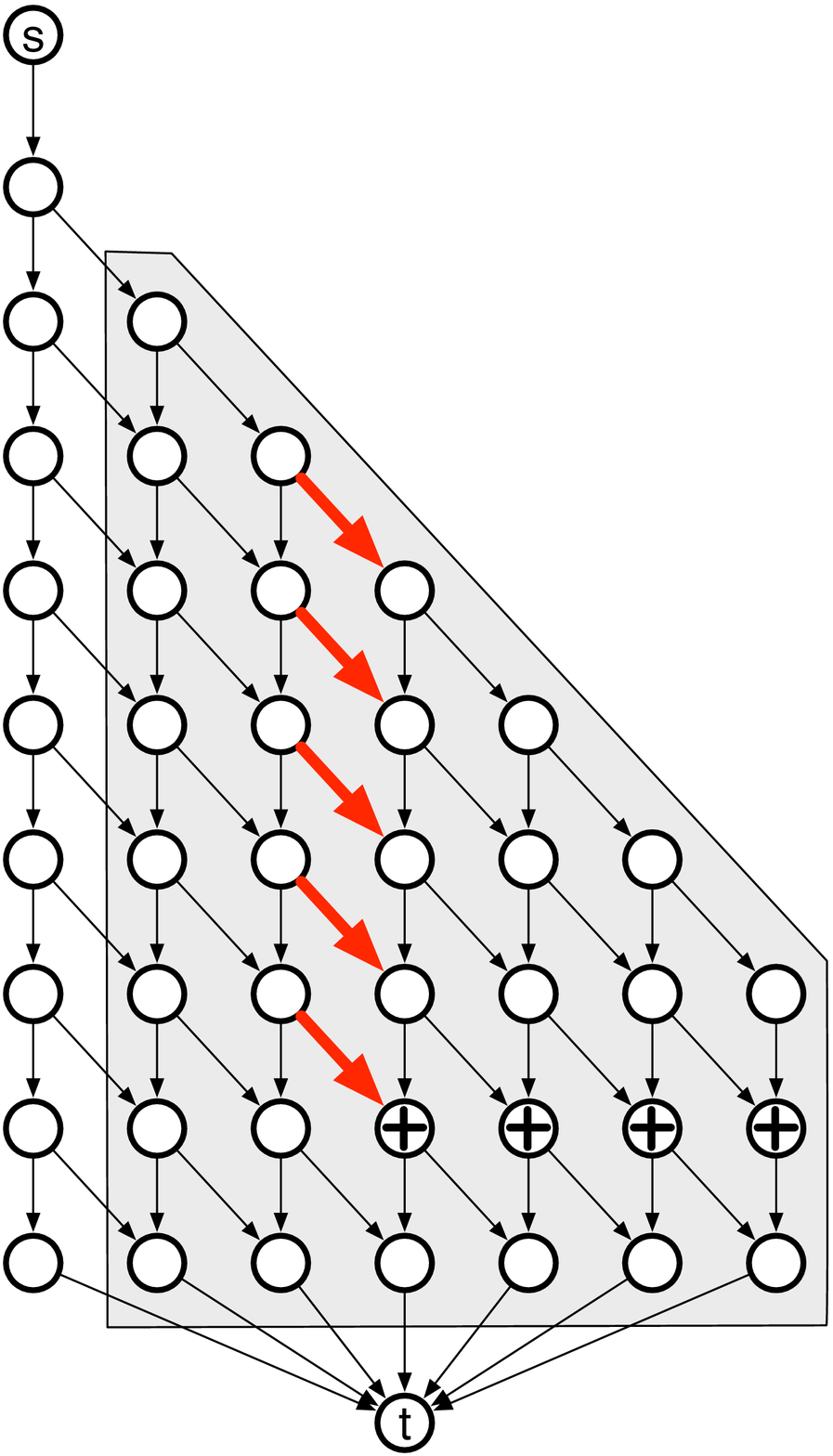}
    \caption{Coefficient vectors of inequalities~\eqref{eq:bindxy} (left) and~\eqref{eq:bindxybar} (right).}\label{fig:eq56}
\end{figure}
Since $\flows{p}{q}$ is the set of all $s$-$t$-flows of value one, all $(x,y)\in\extpoly{p}{q}$ satisfy $\zerovec\le y\le\onevec$, and thus $\zerovec\le x\le\onevec$ due to~\eqref{eq:bindxy} and~\eqref{eq:bindxybar}.  Furthermore, inequalities~\eqref{eq:bindxybar} imply the row-sum inequalities $x(\row{i})\le 1$ for all $(x,y)\in\extpoly{p}{q}$.

\begin{theorem}\label{thm:integral}
    The polytope $\extpoly{p}{q}$ is integral.
\end{theorem}

For the proof of this theorem, we need the following result.

\begin{lemma}\label{lem:integral}
    Let $x_1,\dots,x_n\ge 0$ and $y_1,\dots,y_n\in\R$ with
    \[
        \sum_{\ell=j}^nx_{\ell}\le\sum_{\ell=j}^ny_{\ell}
        \quad\text{for all }1\le j\le n\,.
    \]
    For all numbers $\alpha_1,\dots,\alpha_n\in\R$ and $0\le\beta_1\le\beta_2\le\cdots\le\beta_n$ with
    \[
        \alpha_j\le\beta_{j}
        \quad\text{for all }1\le j\le n
    \]
    the inequality
    \[
        \sum_{j=1}^n\alpha_jx_j\le\sum_{j=1}^n\beta_jy_j
    \]
    holds.
\end{lemma}

\begin{proof}[Proof of Lemma~\ref{lem:integral}]
    We prove the claim by induction on~$n$. The case $n=1$ is trivial, thus let $n\ge 1$.
    Ignoring index~$1$ and decreasing the remaining $\alpha_j$ and $\beta_j$ by  $\beta_1$, from the induction hypothesis we obtain
    \[
        \sum_{j=2}^n(\alpha_j-\beta_1)x_j\le\sum_{j=2}^n(\beta_j-\beta_1)y_j\,.
    \]
    Due to $\alpha_1\le\beta_1$ and $x_1\ge 0$ we have $(\alpha_1-\beta_1)x_1\le 0$, thus we deduce
    \[
    \sum_{j=1}^n(\alpha_j-\beta_1)x_j\le\sum_{j=1}^n(\beta_j-\beta_1)y_j\,.
    \]
    Hence we have
    \[
        \sum_{j=1}^n\alpha_jx_j-\sum_{j=1}^n\beta_jy_j\le
        \beta_1\cdot\Big(\sum_{j=1}^nx_j-\sum_{j=1}^ny_j\Big)\,,
    \]
    with a nonpositive right-hand side due to $\beta_1\ge 0$ and $\sum_{j=1}^nx_j\le\sum_{j=1}^ny_j$.
\end{proof}

\begin{proof}[Proof of Theorem~\ref{thm:integral}]
    In order to show that~$\extpoly{p}{q}$ is integral, we show that for an arbitrary objective function vector $c\in\R^{\orbipartinds{p}{q}}\times\R^{\arcs{p}{q}}$ the optimization problem
    \begin{equation}\label{eq:intproof:optprob}
        \max\setdef{\scalprod{c}{(x,y)}}{(x,y)\in\extpoly{p}{q}}
    \end{equation}
    has an optimal solution with 0/1-components.

    We define two vectors $c^{(1)},c^{(2)}\in\R^{\orbipartinds{p}{q}}\times\R^{\arcs{p}{q}}$ which are zero in all components with the following exceptions:
    \[
    \begin{array}{lcll}
        c^{(1)}_{\darc{i-1}{j-1}} & = & c_{(i,j)} & \text{for all }(i,j)\in\orbipartinds{p}{q}\\
        c^{(2)}_{\varc{i-1}{j}}   & = &
                \max\{0,c_{(i,1)},\dots,c_{(i,j)}\} &
                \text{for all }(i,j)\in\orbipartinds{p}{q}\\
    \end{array}
    \]
    We are going to establish the following two claims:
    \begin{enumerate}
        \item For each $(x,y)\in\extpoly{p}{q}$ we have
        \[
            \scalprod{c}{(x,y)}\le\scalprod{c+c^{(1)}+c^{(2)}}{(\zerovec,y)}\,.
        \]
        \item For each $s$-$t$-flow $y\in\flows{p}{q}\cap\{0,1\}^{\arcs{p}{q}}$ there is some $x\in\{0,1\}^{\orbipartinds{p}{q}}$ with
        \[
            (x,y)\in\extpoly{p}{q}
            \quad\text{and}\quad
            \scalprod{c+c^{(1)}+c^{(2)}}{(\zerovec,y)}=\scalprod{c}{(x,y)}\,.
        \]
    \end{enumerate}
    With these two claims, the existence of an integral optimal solution to~\eqref{eq:intproof:optprob} can be established as follows: Let $\tilde{c}\in\R^{\arcs{p}{q}}$ be the $y$-part of $c+c^{(1)}+c^{(2)}$.
    As $\flows{p}{q}$ is a 0/1-polytope, there is a 0/1-flow $y^{\star}\in\flows{p}{q}\cap\{0,1\}^{\arcs{p}{q}}$ with
    \[
    \scalprod{\tilde{c}}{y^{\star}}=\max\setdef{\scalprod{\tilde{c}}{y}}{y\in\flows{p}{q}}\,.
    \]
    Due to $\scalprod{c+c^{(1)}+c^{(2)}}{(\zerovec,y)}=\scalprod{\tilde{c}}{y}$ for all $y\in\flows{p}{q}$, claim~(1) implies that the optimal value of~\eqref{eq:intproof:optprob} is at most $\scalprod{\tilde{c}}{y^{\star}}$. On the other  hand, claim~(2) ensures that there is some $x^{\star}\in\{0,1\}^{\orbipartinds{p}{q}}$ with $(x^{\star},y^{\star})\in\extpoly{p}{q}$ and
    \[
        \scalprod{c}{(x^{\star},y^{\star})}
        =\scalprod{c+c^{(1)}+c^{(2)}}{(\zerovec,y^{\star})}
        =\scalprod{\tilde{c}}{y^{\star}}\,.
    \]
Thus, $(x^{\star},y^{\star})$ is an integral optimal solution to~\eqref{eq:intproof:optprob}.

In order to prove claim~(2), let $y\in\flows{p}{q}\cap\{0,1\}^{\arcs{p}{q}}$, i.e., $y$ is the incidence vector of an $s$-$t$-path in $\digraph{p}{q}$. For the construction of some $x\in\{0,1\}^{\orbipartinds{p}{q}}$ as required we start by initializing $x=\zerovec$. For each $(i,j)\in\nodes{p}{q}$ with $y_{\darc{i}{j}}=1$ set $x_{i+1,j+1}=1$. For every $(i,j)\in\nodes{p}{q}$ with $j\ge 1$ and $y_{\varc{i}{j}}=1$   choose $\ell\in\ints{j}$ with
\[
    c_{i+1,\ell}=\max\{c_{(i+1),1},\dots,c_{(i+1),j}\}\,,
\]
and set $x_{i+1,\ell}=1$ if $c_{i+1,\ell}\ge 0$.

For claim~(1) let $(x,y)\in\extpoly{p}{q}$. Define $x'\in\R^{\orbipartinds{p}{q}}$ via
\[
    x'_{ij}=x_{ij}-y_{\darc{i-1}{j-1}}
\]
for all $(i,j)\in\orbipartinds{p}{q}$. As $(x,y)$ satisfies~\eqref{eq:bindxy},  $x'\ge\zerovec$ holds. Furthermore, we have
\begin{equation}\label{eq:inproof:1}
    \scalprod{c^{(1)}}{(\zerovec,y)}=\scalprod{c}{(x-x',\zerovec)}\,.
\end{equation}
Therefore, it suffices to show
\begin{equation}\label{eq:inproof:2}
    \scalprod{c^{(2)}}{(\zerovec,y)}\ge\scalprod{c}{(x',\zerovec)}\,,
\end{equation}
because~\eqref{eq:inproof:1} and~\eqref{eq:inproof:2} yield
\begin{eqnarray*}
    \scalprod{c+c^{(1)}+c^{(2)}}{(\zerovec,y)}
    & =   &  \scalprod{c}{(\zerovec,y)}
            +\scalprod{c^{(1)}}{(\zerovec,y)}
            +\scalprod{c^{(2)}}{(\zerovec,y)}\\
    & \ge & \scalprod{c}{(\zerovec,y)}+\scalprod{c}{(x-x',\zerovec)}+\scalprod{c}{(x',\zerovec)}\\
    & =   & \scalprod{c}{(x,y)}\,.
\end{eqnarray*}

In order to establish~\eqref{eq:inproof:2}, we prove for every $i\in\ints{p}$
\begin{equation}\label{eq:inproof:3}
    \sum_{j=1}^{q(i)}c_{ij}x'_{ij}\le\sum_{j=1}^{q(i-1)}c^{(2)}_{\varc{i-1}{j}}y_{\varc{i-1}{j}}\,,
\end{equation}
which by summation over $i\in\ints{p}$
yields~\eqref{eq:inproof:2}. To see \eqref{eq:inproof:3} for some
$i\in\ints{p}$, observe that, for every $j\in\ints{q(i)}$, we have
\[
    x'(\barr{i}{j})=x(\barr{i}{j})-y(\diaginarcs{\barr{i}{j}})
    =x(\barr{i}{j})-y(\barr{i}{j})+y(\vertinarcs{\barr{i}{j}})\,.
\]
Due to Remark~\ref{rem:cuts}, and as $(x,y)$
satisfies~\eqref{eq:bindxybar}, this implies
\[
    x'(\barr{i}{j})\le y(\vertinarcs{\barr{i}{j}})\,.
\]
Defining $y_{\varc{i}{q(i)}}=0$ in case of $q(i-1)<q(i)$, we thus
have
\[
    \sum_{\ell=j}^{q(i)}x'_{i\ell}\le\sum_{\ell=j}^{q(i)}y_{\varc{i-1}{\ell}}
\]
for every $j\in\ints{q(i)}$.
Setting  $c^{(2)}_{\varc{i-1}{q(i)}}$ to the biggest component of~$c$ in case of $q(i-1)<q(i)$, we furthermore have
\[
    0\le c^{(2)}_{\varc{i-1}{1}}\le\cdots\le  c^{(2)}_{\varc{i-1}{q(i)}}
\]
and  $c_{ij}\le c^{(2)}_{\varc{(i-1)}{j}}$ for all $j\in\ints{q(i)}$. Thus we can use Lemma~\ref{lem:integral} (with $n=q(i)$, $x_j=x'_{ij}\ge 0$, $y_j=y_{\varc{i-1}{j}}$, $\alpha_j=c_{ij}$, and $\beta_j=c^{(2)}_{\varc{i-1}{j}}$) to deduce
\[
\sum_{j=1}^{q(i)}c_{ij}x'_{ij}\le\sum_{j=1}^{q(i)}c^{(2)}_{\varc{i-1}{j}}y_{\varc{i-1}{j}}\,,
\]
which yields~\eqref{eq:inproof:3}.
\end{proof}

From Theorem~\ref{thm:integral} one obtains that $\extpoly{p}{q}$ is an extended formulation for $\orbipack{p}{q}$.

\begin{theorem}\label{thm:extform}
    The orbitope $\orbipack{p}{q}\subseteq\R^\orbipartinds{p}{q}$ is the orthogonal projection of the polytope $\extpoly{p}{q}\subseteq\R^{\orbipartinds{p}{q}}\times\R^{\arcs{p}{q}}$ to the space  $\R^{\orbipartinds{p}{q}}$.
\end{theorem}

\begin{proof}
    Let $x\in\orbipack{p}{q}\cap\{0,1\}^{\orbipartinds{p}{q}}$ be an arbitrary vertex of~$\orbipack{p}{q}$.
     The incidence vector $y\in\{0,1\}^{\arcs{p}{q}}$ of the unique $s$-$t$-path using all arcs $\darc{i-1}{j-1}$ with $(i,j)\in\orbipartinds{p}{q}$ and $x_{ij}=1$  satisfies $(x,y)\in\extpoly{p}{q}$.
    Thus, $\orbipack{p}{q}$ is contained in the projection of $\extpoly{p}{q}$.

    To see that vice versa the projection of $\extpoly{p}{q}$ is contained in $\orbipack{p}{q}$, by Theorem~\ref{thm:integral} it suffices to observe that every 0/1-point $(x,y)\in\extpoly{p}{q}$ is contained in $\orbipack{p}{q}$. Clearly, for such a point $x$ has at most one one-entry per row (since the row-sum inequalities are implied by the fact  $(x,y)\in\extpoly{p}{q}$). Furthermore, if the $j$-th column of~$x$ was lexicographically larger than the $(j-1)$-st column of~$x$ with $i$ being minimal such that $x_{ij}=1$ holds,
then one would find that $y(\vbar{i}{j-1})=0$ holds (because of~\eqref{eq:bindxy}), contradicting~\eqref{eq:bindxybar} for $(i,j-1)$.
\end{proof}

From the proof of Theorem \ref{thm:integral}, we derive a
combinatorial algorithm for the linear optimization problem
\begin{equation}\label{optpack}
    \max\setdef{\scalprod{d}{x}}{x\in\orbipack{p}{q}}
\end{equation}
with $d\in\orbipartinds{p}{q}$.
Indeed, with $c=(d,\zerovec)\in\R^{\orbipartinds{p}{q}}\times\R^{\arcs{p}{q}}$ every optimal solution $(x^{\star},y^{\star})$ to
\begin{equation}\label{optpackext}
    \max\setdef{\scalprod{c}{(x,y)}}{(x,y)\in\extpoly{p}{q}}
\end{equation}
yields an optimal solution~$x^{\star}$ to~\eqref{optpack}.
From the proof of Theorem~\ref{thm:integral} we know that we can compute an optimal solution~$(x^{\star},y^{\star})$ to~\eqref{optpackext} by first computing the incidence vector $y^{\star}\in\{0,1\}^{\arcs{p}{q}}$ of a longest $s$-$t$-path in the digraph $\digraph{p}{q}$ with respect to arc length given by $c^{(1)}+c^{(2)}$ (which can be done in linear time since $\digraph{p}{q}$ is acyclic) and then setting~$x^{\star}\in\{0,1\}^{\orbipartinds{p}{q}}$ as described in the proof of claim~(2) (in the proof of Theorem \ref{thm:integral}).

%
%
\begin{corollary}\label{cor:alg}
Linear optimization  over $\orbipack{p}{q}$ can be solved
in time $\bigo{pq}$.
\end{corollary}

\subsection{The partitioning case}
\label{subsec:ext:part}

The previous results can be easily extended to the partitioning
case. Since the row-sum inequalities $x(\row{i})\le 1$ are valid for $\extpoly{p}{q}$,
\[
    \extpolypart{p}{q}=
    \setdef{(x,y) \in \extpoly{p}{q}}%
           {x(\row{i})=1\text{ for all } i \in [p]}
\]
is a face of~$\extpoly{p}{q}$. Clearly, due to Theorem~\ref{thm:extform} this face maps to the face (see Sect.~\ref{sec:intro})
\[
    \setdef{x\in\orbipack{p}{q}}{x(\row{i})=1\text{ for all } i \in [p]}=\orbipart{p}{q}
\]
 of $\orbipack{p}{q}$ via the orthogonal projection onto the $x$-space ~$\R^{\orbipartinds{p}{q}}$.

\begin{corollary}\label{cor:extpartpath}
$\extpolypart{p}{q}$ is an extended formulation for
$\orbipart{p}{q}$.
\end{corollary}

Suppose we want to solve
\begin{equation}\label{optpart}
    \max\setdef{\scalprod{d}{x}}{x\in\orbipart{p}{q}}
\end{equation}
for some $d\in\R^{\orbipartinds{p}{q}}$. As all points $x\in\orbipart{p}{q}$ satisfy the row-sum equations $x(\row{i})=1$ for all $i\in\ints{p}$, we may add, for each $i$, an arbitrary constant to the objective function coefficients of the variables belonging to $\row{i}$ without changing the optimal solutions to~\eqref{optpart} (though, of course, changing the objective function values of the solutions). Therefore, we may assume that~$d$ has only positive components. But then
\[
    \max\setdef{\scalprod{d}{x}}{x\in\orbipart{p}{q}}=\max\setdef{\scalprod{d}{x}}{x\in\orbipack{p}{q}}\,,
\]
and all optimal solutions to the optimization problem over~$\orbipack{p}{q}$ are points in~$\orbipart{p}{q}$. Thus we derive the following from Corollary~\ref{cor:alg}.
\begin{corollary}\label{cor:alg:part}
Linear optimization  over $\orbipart{p}{q}$ can be solved
in time $\bigo{pq}$.
\end{corollary}


\subsection{Reducing the number of variables and nonzero elements}
\label{subsec:ext:reduce}

Let us manipulate the defining system of $\extpoly{p}{q}$ in order
to decrease the number of variables and nonzero coefficients. This may be advantageous for practical purposes. It furthermore emphasizes the simplicity of the extended formulation. For the
sake of readability, we define $\barr{i}{j}=\varnothing$ and
$\vbar{i}{j}=\varnothing$ whenever $j>q(i)$.

Since  every $y\in\flows{p}{q}$ satisfies  $y_{\varc{i}{j}} = y(\barr{i}{j}) -
y(\barr{i+1}{j+1})$, we deduce from
Remark~\ref{rem:cuts} that for all $(x,y)\in\extpoly{p}{q}$
\begin{equation*}
        y_{\varc{i}{j}} =y(\diaginarcs{\vbar{i}{j}}) -
        y(\diaginarcs{\vbar{i+1}{j+1}})
\end{equation*}
holds for all vertical arcs $\varc{i}{j}$ with $(i,j)\in\orbipartinds{p}{q}$ (and $i<p$) as well as
\begin{equation*}
y_{((p,j),t)}=y(\diaginarcs{\vbar{p}{j}})-
y(\diaginarcs{\vbar{p}{j+1}})
\end{equation*}
for all arcs $((p,j),t)$ with $j\in\intszero{q}$, where we defined
$y(\diaginarcs{\vbar{p}{q+1}})=0$. Similarly to the derivation of
Remark~\ref{rem:cuts}, one furthermore deduces that every
$(x,y)\in\extpoly{p}{q}$ satisfies
\begin{equation*}
    y_{\varc{i}{0}}=1-y(\diaginarcs{\vbar{i+1}{1}})
\end{equation*}
for all $i\in\intszero{p-1}$.
Finally, every $(x,y)\in\extpoly{p}{q}$ clearly satisfies
    $y_{(s,(0,0))}=1$.
Therefore, we can eliminate from the
system describing~$\extpoly{p}{q}$ all arc variables except for the ones corresponding to diagonal arcs.

We finally apply the linear transformation defined by
\[
    z_{ij}=x(\barr{i}{j}) \quad\text{and}\quad w_{ij}=y(\diaginarcs{\vbar{i}{j}})\quad\hbox{ for all } (i,j)\in\orbipartinds{p}{q}
\,,
\]
to $\R^{\orbipartinds{p}{q}}\times\R^{\diagarcs{p}{q}}$,
whose inverse is given by
\[
    x_{ij}=z_{i,j}-z_{i,j+1} \quad\text{for all }(i,j)\in\orbipartinds{p}{q}\,.
\]
(defining $z_{i,q(i)+1}=0$, for all $i \in [p]$) and
\[
y_{\darc{i}{j}}=w_{i+1,j+1}-w_{i,j+1}\quad\text{for all } i \in [p-1]_0, j \in
[q(i+1)-1]_0\,.
\]

Few calculations are needed to check that the previous
transformation (bijectively) maps $\extpoly{p}{q}$ onto the polytope $\extpolycompact{p}{q}\subseteq\R^{\orbipartinds{p}{q}}\times\R^{\orbipartinds{p}{q}}$ defined by
the following ''very compact'' set of constraints:
\begin{eqnarray}\label{NEW}
w_{i+1,j+1} - w_{i,j+1} & \geq 0 & \hbox{for }    i \in [p-1]_0, j \in[q(i+1)-1]_0 \label{NEF:F}\\
 w_{i,j} - w_{i+1,j+1} & \geq  0 & \hbox{for } (i,j)\in\orbipartinds{p}{q}, i<p \label{NEF:G}\\
 w_{p,1} & \leq  1 \label{NEF:H}\\
 w_{i,j} -w_{i-1,j} - z_{ij}+z_{i,j+1} &\le 0 & \hbox{for } (i,j)\in\orbipartinds{p}{q}  \label{NEF:I}\\
 z_{i,j} -  w_{i,j}& \leq 0 & \hbox{for } (i,j)\in\orbipartinds{p}{q} \label{NEF:J}\\
 w_{i,q(i)} & \geq  0& \hbox{for } i \in [p] \label{NEF:L}
\end{eqnarray}
Here, \eqref{NEF:F}  represent the nonnegativity constraints on the diagonal arcs. Nonnegativity on the vertical arcs $\varc{i}{j}$ with $i\in\intszero{p-1}$  is reflected by~\eqref{NEF:G} for $j\in\ints{q(i)}$ and by~\eqref{NEF:H} (together with the nonnegativity of~$w$, which is implied by~\eqref{NEF:L} and~\eqref{NEF:G}) for $j=0$. Finally, equations~\eqref{eq:bindxy} and~\eqref{eq:bindxybar} translate to~\eqref{NEF:I} and~\eqref{NEF:J}, respectively.

Ignoring the nonnegativity constraints,  system~\eqref{NEF:F}--\eqref{NEF:L} has less than $2pq$ variables and $4pq$ constraints,
for a total number of nonzero coefficients that is smaller than $10pq$.

Note that $w_{1,1}\le 1$ is a valid inequality
for~$\extpolycompact{p}{q}$. The face of~$\extpolycompact{p}{q}$
defined by $w_{1,1}= 1$ is the image of the face
$\extpolypart{p}{q}$ of $\extpoly{p}{q}$. Thus, adding~$w_{1,1}=
1$ to the system~\eqref{NEF:F}--\eqref{NEF:L} one arrives at
another extended formulation of $\orbipart{p}{q}$.

We summarize the results of this subsection.
\begin{theorem}\label{thm:verycompact}
    The polytope $\extpolycompact{p}{q}\subseteq\R^{\orbipartinds{p}{q}}\times\R^{\orbipartinds{p}{q}}$ defined by~\eqref{NEF:F}--\eqref{NEF:L} is an extended formulation of~$\orbipack{p}{q}$. The face of~$\extpolycompact{p}{q}$ defined by $w_{1,1}= 1$ is an extended formulation of~$\orbipart{p}{q}$.
\end{theorem}


\section{The projection}
\label{sec:project}

Let $\scipoly{p}{q}\subseteq\R^{\orbipartinds{p}{q}}$ be the polytope defined by the nonnegativity constraints $x\ge\zerovec$, the row-sum inequalities
$x(\row{i})\le 1$ for all $i\in\ints{p}$
and all shifted-column inequalities. By checking the vertices (0/1-vectors) of $\orbipack{p}{q}$ it is easy to see that $\orbipack{p}{q}\subseteq\scipoly{p}{q}$ holds. Thus, in order to prove
\begin{theorem}\label{thm:orbipackSCI}
    $\orbipack{p}{q}=\scipoly{p}{q}$
\end{theorem}
\noindent (which is Prop.~13 in~\cite{KP08}) it suffices (due to
Theorem~\ref{thm:extform}) to show the following:
\begin{theorem}\label{thm:liftscipoly}
    For each $x\in\scipoly{p}{q}$ there is some $y\in\flows{p}{q}$ with $(x,y)\in\extpoly{p}{q}$.
\end{theorem}

\begin{proof}
    For $x\in\scipoly{p}{q}$ consider the network $\digraph{p}{q}$ with
    \[
        \text{capacity }x_{ij}\text{ on the diagonal arc }\darc{i-1}{j-1}
    \]
for each $(i,j)\in\orbipartinds{p}{q}$ and infinite capacities on
all other arcs. In this network, we construct a feasible flow
$y\in\flows{p}{q}$ of value one with the property
\begin{equation}\label{eq:flowcond}
    y_{\varc{i-1}{j-1}}>0\quad\Rightarrow\quad y_{\darc{i-1}{j-1}}=x_{ij}
\end{equation}
for all $(i,j)\in\orbipartinds{p}{q}$. Phrased verbally, the
flow~$y$ uses a vertical arc only if the diagonal arc emanating
from its tail is saturated. Such a flow can easily be constructed
in the following way: start by sending one unit of flow from~$s$
to~$t$ along the vertical path in column zero. At each  step, if
the flow~$y$ constructed so far violates~\eqref{eq:flowcond} for
some $(i,j)\in\orbipartinds{p}{q}$, choose such a pair $(i,j)$
with minimal~$j$, breaking ties by choosing~$i$ minimally as well.
With
\[
    \vartheta=\min\{y_{\varc{i-1}{j-1}},x_{ij}-y_{\darc{i-1}{j-1}}\}
\]
(i.e., the minimum of the flow on the vertical arc and the residual capacity on the diagonal arc starting at $(i-1,j-1)$) reroute~$\vartheta$ units of the flow  currently travelling on the vertical arc $\varc{i-1}{j-1}$ along the path starting with the diagonal arc  $\darc{i-1}{j-1}$ and then using the vertical arcs in column~$j$.
Note that this affects only arcs leaving nodes $(k,\ell)$ with $k\ge i$ and $\ell\ge j$.
 After this rerouting, \eqref{eq:flowcond} holds for $(i,j)$. The minimality requirements in the choice of $(i,j)$ ensure that the flow on the two arcs leaving $(i,j)$ is not changed again afterwards. Thus, \eqref{eq:flowcond} will always be satisfied for $(i,j)$ in the future. Therefore, the procedure eventually ends with a flow as required.

As $(x,y)$ satisfies~\eqref{eq:bindxy} by construction, it  suffices to show~\eqref{eq:bindxybar} in order to prove
 $(x,y)\in\extpoly{p}{q}$.
To this end, let $(i,j)\in\orbipartinds{p}{q}$. Due to Remark~\ref{rem:cuts}, we only need to prove
\begin{equation}\label{eq:barbar}
    y(\barr{i}{j})\ge x(\barr{i}{j})\,.
\end{equation}

We construct a directed $(k,\ell)$-$(i,j)$-path~$\Gamma$ in the residual network with respect to the flow~$y$ (containing only those arcs of~$\digraph{p}{q}$ that are not saturated by~$y$)
with $k=\ell$ or $\ell=0$ in the following way: Starting from the trivial (length zero) $w$-$(i,j)$-path with $w=(i,j)$, in each step we extend the path at its current start node~$w$ by the diagonal arc entering~$w$ if this arc is part of the residual network, and by the vertical arc entering~$w$ otherwise. As the residual network contains all vertical arcs, we clearly can proceed this way until the start node of the current path is some node $(k,\ell)$ with $k=\ell$ or with $\ell=0$.

Since~$\Gamma$ is a path in the residual network and due to~\eqref{eq:flowcond}, we have
\begin{equation}\label{eq:vertouttgamma}
    y(\vertoutarcs{\tnodeset{\Gamma}})=0\,.
\end{equation}

If $\ell=0$, part~(2) of Lemma~\ref{lem:cuts} together with~\eqref{eq:vertouttgamma} yields $y(\barr{i}{j})=1$, from which~\eqref{eq:barbar} follows since~$x$ satisfies the row-sum inequality $x(\row{i})\le 1$ and the nonnegativity constraints.

If $\ell\ne 0$, then $k=\ell\ge 1$. Thus, according to part~(1) of Lemma~\ref{lem:cuts} and due to~\eqref{eq:vertouttgamma}, we have
\begin{equation}\label{eq:diaginsbarrr}
    y(\diaginarcs{\snodeset{\Gamma}})=y(\barr{i}{j})\,.
\end{equation}
Since we preferred diagonal arcs from the residual network in our backwards construction of~$\Gamma$, we find that all arcs from $\diaginarcs{\snodeset{\Gamma}}$ are saturated by the flow~$y$. Therefore, for the shifted column $S=\snodeset{\Gamma}$, we have (using~\eqref{eq:diaginsbarrr})
\begin{equation}\label{eq:scdiaginsnodes}
    x(S)=y(\diaginarcs{\snodeset{\Gamma}}=y(\barr{i}{j})\,.
\end{equation}
Let $a\in\arcs{p}{q}$ be the arc in~$\Gamma$ entering~$(i,j)$, and denote by~$\Gamma'$ the path arising from~$\Gamma$ by removing~$a$.

If~$a$ is  diagonal, then the $(\ell,\ell)$-$(i-1,j-1)$-path~$\Gamma'$ satisfies $\snodeset{\Gamma'}=S$, and the shifted-column inequality $x(\barr{i}{j})\le x(S)$ (satisfied by~$x$) establishes~\eqref{eq:barbar} via~\eqref{eq:scdiaginsnodes}.

If~$a$ is  vertical, by construction of~$\Gamma$, we have $y_{\darc{i-1}{j-1}}=x_{ij}$. Furthermore, the $(\ell,\ell)$-$(i-1,j)$-path~$\Gamma'$ satisfies $\snodeset{\Gamma'}=S\setminus\{(i,j\})$. Thus  using the shifted-column inequality $x(\barr{i}{j+1})\le x(S\setminus\{(i,j)\})$ one obtains from~\eqref{eq:scdiaginsnodes} the inequality
\[
    y(\barr{i}{j})=x(S)=x(S\setminus\{(i,j)\})+x_{ij}\ge x(\barr{i}{j+1})+x_{ij}=x(\barr{i}{j})\,.
\]
Thus, \eqref{eq:barbar} is established also in this case, which finally proves Theorem~\ref{thm:liftscipoly}.
\end{proof}

\section{Remarks}
\label{sec:remarks} In our view, the extended formulations for the
orbitopes~$\orbipack{p}{q}$ and $\orbipart{p}{q}$ presented in
this paper once more demonstrate the power that lies in the
concept of extended formulations. Not only do the extended
formulations provide a very compact way of describing the
orbitopes, but also do they allow to derive rather simple proofs
of the fact that nonnegativity constraints, row-sum
inequalities/equations, and SCIs suffice in order to linearly
describe~$\orbipack{p}{q}$ and~$\orbipart{p}{q}$.

To us it seems that these proofs better reveal the reason why  SCIs are necessary and (basically) sufficient in these descriptions. The construction of the flow in the proof of Thm.~\ref{thm:liftscipoly} is quite natural. The rest of the proof (i.e., the backwards construction of the path~$\Gamma$) one may  also have  done without knowing the SCIs in advance. Thus, knowing the extended formulation, one possibly could also have \emph{detected} SCIs on the way trying to do this proof.

An interesting practical question is whether the very sparse and compact extended formulations for orbitopes lead to performance gains in branch-and-cut algorithms compared to versions that dynamically add SCIs via the linear time separation algorithm (described in~\cite{KP08}).

\subsection*{Acknowledgements}
We would like to thank Laura Sanit\`{a} for useful discussions and Marc Pfetsch for valuable comments on an earlier version of this paper.

\providecommand{\bysame}{\leavevmode\hbox to3em{\hrulefill}\thinspace}
\providecommand{\MR}{\relax\ifhmode\unskip\space\fi MR }
\providecommand{\MRhref}[2]{%
  \href{http://www.ams.org/mathscinet-getitem?mr=#1}{#2}
}
\providecommand{\href}[2]{#2}

\end{document}